\documentclass{article}
\usepackage{amsmath, amssymb, fullpage, amsthm, xcolor, enumerate}
\usepackage[normalem]{ulem}
\usepackage[justification=centering]{caption}
\usepackage{indentfirst}

\newtheorem{theorem}{Theorem}[section]
\newtheorem{corollary}[theorem]{Corollary}
\newtheorem{lemma}[theorem]{Lemma}
\newtheorem{proposition}[theorem]{Proposition}
\newtheorem{definition}[theorem]{Definition}

\newtheorem{remark}[theorem]{Remark}
\newtheorem{example}[theorem]{Example}
\newtheorem{notation}[theorem]{Notation}

\def\drop#1{}
\begin{document}

\begin{center}

{\Large\textbf{Super Hayashi Quandles}} \\
\vspace{8px}
{\Large{Ant\'onio Lages\footnote{Supported by FCT LisMath fellowship PD/BD/145348/2019.}\qquad Pedro Lopes\footnote{Supported by project FCT UIDB/04459/2020; member of CAMGSD; corresponding author.}}} \\
\vspace{8px}
{\small{Department of Mathematics}} \\
{\small{Instituto Superior T\'ecnico}} \\
{\small{Universidade de Lisboa}} \\
{\small{Av. Rovisco Pais, 1049-001 Lisbon, Portugal}} \\
{\small{\texttt{\{antonio.lages,pedro.f.lopes\}@tecnico.ulisboa.pt}}} \\
\vspace{4px}
\end{center}

\pagenumbering{arabic}

\pagestyle{plain}

\begin{abstract}
Quandles are right-invertible, right-self distributive (and idempotent) algebraic structures. Therefore,  right translations are automorphisms.
For each finite quandle, the list of the right translations' cycle structures is known as the profile of the quandle. For a connected quandle, any two right translations are conjugate so there is essentially one cycle structure per connected quandle - which we thus identify with the profile. Hayashi conjectured that, for a finite connected quandle, each length of its profile divides the longest length. In the present article we introduce Super Hayashi Quandles (SHQ). An SHQ is a finite connected quandle such that any two lengths in its profile are (i) distinct, and (ii) the shorter one divides the longer one. The SHQ's are latin quandles and we prove that their profiles depend only on the second shortest length and on the number of cycles, along with an explicit description of each of the lengths. Furthermore, we prove that SHQ's have SHQ's alone for subquandles (with the same second shortest length but fewer cycles). Finally, we construct infinitely many SHQ's.
\\
{\textcolor{white}{.}}\\
\textbf{Keywords}: quandle; right translation; cycle; profile; Hayashi conjecture; super Hayashi quandle.\\
\textbf{MSC2020}: 20N99.

\end{abstract}

\section{Introduction}

A quandle is a right-invertible,  right self-distributive and idempotent algebraic structure (\cite{Joyce}, \cite{Matveev}). These axioms constitute an algebraic counterpart to the Reidemeister moves (\cite{Kauffman}) of Knot Theory and quandles give rise to efficient ways of telling knots apart (for instance, via counting colorings, \cite{Dionisio_Lopes}, \cite{Bojarczuk_Lopes}). We are thus led to study the structure of quandles.\par

The right-invertibility and right self-distributivity axioms imply  that any right translation (multiplication by a fixed element on the right) is an automorphism. It has been interesting to look into finite quandles by way of the cycle structures their right translations may have \cite{Lopes_Roseman}, \cite{Hayashi}, \cite{Kajiwara_Nakayama}, \cite{KamadaTamaruWada}, \cite{Lages_Lopes_Cyclic}, \cite{Lages_Lopes_Vojtěchovský}, \cite{Lages_Lopes}, \cite{Watanabe1}, \cite{Watanabe2}, \cite{Watanabe3}. For each quandle, the list of these cycle structures is known as the profile of the quandle. For a connected quandle, any two right translations are conjugate so they have essentially one cycle structure - which we identify with its profile. In \cite{Hayashi}, Hayashi conjectured that, for a finite connected quandle, any length from its profile divides the longest length. In the present article we introduce  Super Hayashi Quandles (SHQ) which are quandles that satisfy Hayashi Conjecture by definition. An SHQ is a finite quandle such that   the cycle structure of each of its right translations satisfies the following condition: any two distinct disjoint cycles are such that (i) they have distinct lengths, and (ii) the shorter length divides the longer one. By Vojt\v{e}chovsk\'y's and the authors' previous work (cf. \cite{Lages_Lopes_Vojtěchovský}) (i) above implies that such quandles are latin, thus connected. We prove that the profiles of SHQ's are characterized by two positive integer parameters, $c$ (the number of cycles) and $\ell$ (the length of the second smallest cycle) - this is parts 1. and 2. of the Main Theorem, below. In particular, we prove that the profile of each such quandle is $(1,\ell,\ell(\ell+1),\dots,\ell(\ell+1)^{c-2})$, where $\ell+1$ is a prime power. Furthermore, we prove that SHQ's have SHQ's alone for subquandles (with the same second shortest length but fewer cycles), one per $c'$, the number of cycles, $2\leq c' < c$ - see part 3. of the Main Theorem, below. Finally, we construct infinitely many SHQ's.\par

The article is organized as follows. In Section \ref{sctn:background} we lay out the background material. In Section \ref{sctn:nested} we define Super Hayashi Quandles, we set up the notation and we state the Main Theorem. In Section \ref{sctn:aux} we prove some auxiliary results. In Section \ref{sctn:characterization} we prove the Main Theorem. In Section \ref{sctn:final} we present infinitely many families of infinitely many Super Hayashi quandles. We also make remarks for further study.

The authors thank an anonymous referee of a previous version of this article for helpful comments.

\section{Quandles and their Profiles}\label{sctn:background}

We first define the algebraic structure known as \emph{quandle}, introduced independently in \cite{Joyce} and \cite{Matveev}.

\begin{definition}[Quandle]\label{def:quandle}
Let $X$ be a set equipped with a binary operation denoted $*$. The pair $(X,*)$ is said to be a $\emph{quandle}$ if, for each $i, j, k\in X$,
\begin{enumerate}
\item $i*i=i$ (idempotency);
\item $\exists ! x\in X: x*j=i$ (right-invertibility);
\item $(i*j)*k=(i*k)*(j*k)$ (right self-distributivity).
\end{enumerate}
\end{definition}

Alternatively, a quandle can be regarded as described in  Theorem \ref{thm:equivdef}.

\begin{theorem}\label{thm:equivdef}
Let $X$ be a set of order $n$ and suppose a permutation $R_i\in S_n$ is assigned to each $i\in X$. Then the expression $j*i:= R_i(j), \forall  j \in X$, yields a quandle structure on $X$ if and only if $R_{ R_i(j)}=  R_i R_j R_i^{-1}$ and $R_i(i)=i$, $\forall i, j\in X$. This quandle structure is uniquely determined by the set of $n$ permutations.
\end{theorem}

\begin{proof}
The proof is straightforward. The details can be found in \cite{Brieskorn}.
\end{proof}

To each quandle $(X,\ast)$ of order $n$ we associate a \emph{quandle table}. This is an $n\times n$ table whose element in row $i$ and column $j$ is $i\ast j$, for every $i,j\in X$. In the sequel, quandle tables will have an extra $0$-th column (where we display the $i$'s) and an extra $0$-th row (where we display the $j$'s) to improve legibility.

\begin{example}\label{example:Q94}
Table \ref{table:1} is a quandle table for $Q_{9,4}$, a quandle of order $9$, see \cite{Vendramin_2}.

\begin{figure}[htbp]\centering
{\renewcommand{\arraystretch}{1.2}
\renewcommand{\tabcolsep}{6pt}
\begin{tabular}{c|c c c c c c c c c}
    $*$ & 1 & 2 & 3 & 4 & 5 & 6 & 7 & 8 & 9\\
    \hline
    1 & 1 & 3 & 2 & 7 & 8 & 9 & 4 & 5 & 6\\
    2 & 3 & 2 & 1 & 9 & 6 & 5 & 8 & 7 & 4\\
    3 & 2 & 1 & 3 & 5 & 4 & 7 & 6 & 9 & 8\\
    4 & 5 & 7 & 9 & 4 & 1 & 8 & 2 & 6 & 3\\
    5 & 6 & 4 & 8 & 2 & 5 & 1 & 9 & 3 & 7\\
    6 & 7 & 9 & 5 & 8 & 3 & 6 & 1 & 4 & 2\\
    7 & 8 & 6 & 4 & 3 & 9 & 2 & 7 & 1 & 5\\
    8 & 9 & 5 & 7 & 6 & 2 & 4 & 3 & 8 & 1\\
    9 & 4 & 8 & 6 & 1 & 7 & 3 & 5 & 2 & 9\\
\end{tabular}}
\captionof{table}{Quandle table for $Q_{9,4}$.}\label{table:1}
\end{figure}
\end{example}

For a quandle $(X,*)$ and for $i\in X$, we let $R_i: X\rightarrow X, j\mapsto j*i$ be the \emph{right translation} by $i$ in $(X,*)$. For instance, the right translation by $1$ in $Q_{9,4}$ is $R_1=(1)(2\,\,3)(4\,\,5\,\,6\,\,7\,\,8\,\,9)$. Axioms \textit{2.} and \textit{3.} in Definition \ref{def:quandle} ensure that each right translation in a finite quandle is an automorphism of that quandle.

The group $\langle R_i:i\in X\rangle$ generated by the right translations of a quandle $(X,*)$ is the \emph{right multiplication group}. A quandle $(X,*)$ is \emph{connected} if its right multiplication group acts transitively on $X$. For instance, $Q_{9,4}$ is a connected quandle, by inspection of its quandle table (Table \ref{table:1}).

For a permutation $f$ on a finite set $X$, the \emph{cycle structure} of $f$ is the formal sequence $(1^{c_1}, 2^{c_2}, 3^{c_3},\dots)$, such that $c_m$ is the number of $m$-cycles in the decomposition of $f$ into disjoint cycles. For convenience, we omit entries with $c_i=0$ and we drop $c_i$ whenever $c_i=1$. For example, the right translation by $1$ in $Q_{9,4}$ has cycle structure $(1, 2, 6)$. In this article, since we are dealing with SHQ's, the $c_i$'s will be systematically dropped.

\begin{lemma}\label{lem:connected}
In a connected quandle $(X,*)$, any two right translations have the same cycle structure.
\end{lemma}

\begin{proof}
The right multiplication group of $(X,*)$ acts transitively on $X$ and it is contained in the automorphism group of $(X,*)$.
\end{proof}

Therefore, we define the \emph{profile} of a connected quandle $(X,*)$ to be the cycle structure of any of its right translations, notation: $$prof(X,*) .$$For instance,  by inspection of Table \ref{table:1} above (Example \ref{example:Q94}), $$prof(Q_{9,4})=(1,2,6) .$$

\begin{example}\label{ex:qdlesCyclic}
Quandles of cyclic type \cite{Lopes_Roseman}, \cite{Vendramin_1}. Each of these is a finite quandle, $(X, *)$, such that $$prof(X, *)=(1, |X|-1) .$$
\end{example}

In passing, we remark that we can also define profile for a non-connected quandle. In this case, it will be the list of the cycle structures of the distinct right translations.

\section{SHQ's and the Main Theorem}\label{sctn:nested}

We now introduce the definition of \emph{Super Hayashi Quandles}.

\begin{definition}\label{def:nested}
A \emph{Super Hayashi Quandle} (SHQ) is a quandle each of whose right translations has cycle structure of the form $(\ell_1,\ell_2,\dots,\ell_c)$, where $1=\ell_1<\ell_2<\dots<\ell_c$, with $c\geq2$ and $\ell_{i-1}\mid\ell_i$, for each $i\in\{2,\dots,c\}$.
\end{definition}

\begin{proposition}\label{prop:latin}
Let $(X,*)$ be an SHQ. Then,  $(X,*)$ is latin, thus connected.
\end{proposition}
\begin{proof}
Since an SHQ has a profile such that any two distinct $l_i$'s have distinct lengths, then an SHQ is latin $($Theorem 3.4. in \cite{Lages_Lopes_Vojtěchovský}$)$, hence connected.
\end{proof}

\begin{remark}\label{rmk:Hayashi}
Hayashi Conjecture $($\cite{Hayashi}$)$ states that, for any finite connected quandle, each length of the quandles' profile divides the largest length. Thus, SHQ's satisfy the Hayashi Conjecture.
\end{remark}

\begin{remark}\label{rmk:Lopes+Singh}
Let $(X, *)$ be a finite quandle such that one of its left translations is a bijection and the cycle structure of one of its right translations is $(1, \ell_2, \dots , \ell_c)$ with $c>1$, $1<\ell_2<\ell_3< \dots < \ell_c$ and $\ell_{i-1}\mid \ell_i$. Then, $(X, *)$ is an SHQ (cf. \cite{Lopes_Singh}, Theorem 2.1).
\end{remark}

\begin{remark}\label{rmk:nested}
Let $(X,*)$ be a quandle with $1\in X$. Suppose that each right translation of $(X,*)$ has a unique fixed point and $R_1$ has cycle structure $(\ell_1,\ell_2,\dots,\ell_c)$, where $1=\ell_1<\ell_2<\dots<\ell_c$, with $c\geq2$ and $\ell_{i-1}\mid\ell_i$, for each $i\in\{2,\dots,c\}$. Then, $(X,*)$ is an SHQ (cf. \cite{Lages_Lopes_Vojtěchovský}).
\end{remark}

\begin{example}\label{example:cyclic}
Cyclic quandles, see Example \ref{ex:qdlesCyclic}.

For $|X|>2$, the distinct cycles have distinct lengths hence the quandle is latin (\cite{Lages_Lopes_Vojtěchovský}) hence connected.  These quandles are clearly SHQ's.
\end{example}

\begin{example}\label{example:1}
The quandle $Q_{9,4}$, whose quandle table is presented in Table \ref{table:1}, is a quandle each of whose right translations has cycle structure equal to $(1,2,6)$. Since $1\mid2\mid6$, $Q_{9,4}$ is an SHQ. Example \ref{example:1} complies with Remarks  \ref{rmk:Lopes+Singh}, \ref{rmk:nested}.
\end{example}

We set up notation before stating the main result of this article.

\begin{definition}\label{def:notation}
Let $(X,*)$ be a quandle of order $n$ with profile $(1=\ell_1,\ell_2,\dots,\ell_c)$. We let $X=\{1,\dots,n\}$ and set
\begin{align*}
n_i&:=\sum_{k=1}^i\ell_k,\quad\text{for
 }i\in\{1,\dots,c\};\\
X_i&:=\{1,\dots,n_i\}\subseteq X,\quad\text{for   }i\in\{1,\dots,c\};\\
L_i&:=\{n_i-\ell_i+1,\dots,n_i\}\subseteq X,\quad\text{for   }i\in\{1,\dots,c\}.
\end{align*}
Note that $$n_1=\ell_1 = 1 \qquad X_1=L_1=\{ 1, 2, , \dots , n_1 \} = \{  1  \} ,$$ $$\ell_i = |L_i|\quad\text{for   }i\in\{1,\dots,c\} ,$$ $$n_c=n \qquad X_c=X .$$ Furthermore, the $L_i$'s form a partition of $X=\{1,\dots,n\}$.
\end{definition}

We keep this notation throughout this article.

\begin{lemma}\label{lem:smaller}
Let $(X, *)$ be an SHQ with $prof(X, *)=(\ell_1 < \ell_2 < \cdots < \ell_c)$. Then,
$n_{i-1}<\ell_i$, for each $i\in\{2,\dots,c\}$.
\end{lemma}

\begin{proof}
We prove this result by induction. We have $n_1=\ell_1<\ell_2$. Assume now that $n_{i-1}<\ell_i$, for a certain $i\in\{2,\dots,c-1\}$. Then $n_i=n_{i-1}+\ell_i<2\cdot\ell_i\leq\ell_{i+1}$, where we have used the fact that $\ell_i\mid\ell_{i+1}$ and $\ell_i<\ell_{i+1}$.
\end{proof}

The main result of this article is the following theorem. Notice that this theorem expresses the profile of any SHQ in terms of just two parameters, the number of cycles, $c$, and the length of the second smallest cycle, $\ell$.

\begin{theorem}[Main Theorem]\label{thm:theone}
For $c\geq2$, let $(X,*)$ be an SHQ with $prof(X, *)=(\ell_1,\ell_2,\dots,\ell_c)$. Then, there are (i) a prime number, $p$, and (ii) a positive integer, $a$, such that:

\begin{enumerate}
    \item $|X|=(\ell+1)^{c-1}$, where $\ell+1=p^a$;
    \item $prof(X,*)=(1,\ell\cdot(\ell+1)^0,\ell\cdot(\ell+1)^1,\dots,\ell\cdot(\ell+1)^{c-2})$;
    \item For each $i\in \{ 2, \dots , c \}$, $(X_i, \ast)$ is a subquandle of $(X, \ast)$ with $$prof(X_i,*)=(1,\ell\cdot(\ell+1)^0,\dots,\ell\cdot(\ell+1)^{i-2})$$ Furthermore, these are the only non-trivial subquandles of $(X, \ast)$, up to isomorphism.
\end{enumerate}
\end{theorem}

\begin{remark}
Note that once $2.$ in Theorem \ref{thm:theone} is established, it will follow that $$|X|=1+\sum_{i=0}^{c-2}\ell(\ell+1)^i = 1+\ell\frac{(\ell+1)^{c-1}-1}{\ell+1-1}= (\ell+1)^{c-1} .$$The  proof that $\ell+1=p^a$, for prime $p$ and $a\in \mathbf{Z}^+$ is Corollary \ref{cor:X2}  below.
\end{remark}
\section{Auxiliary Results}\label{sctn:aux}

In order to prove Theorem \ref{thm:theone}, we first prove some auxiliary results regarding finite quandles in general (Proposition \ref{prop:assertions},  Proposition \ref{prop:main}, Proposition \ref{prop:subquandle} and Proposition \ref{prop:Fix}). Proposition \ref{prop:assertions} is already known (\cite{Lopes_Roseman}), it is included here for completeness, as is Proposition \ref{prop:Fix}.  Proposition \ref{prop:main} and Proposition \ref{prop:subquandle} are new results, applicable to quandles in general.

\begin{proposition}\label{prop:assertions}
Let $(X, *)$ be a finite quandle.
Modulo isomorphism, the right multiplications of $(X,*)$ satisfy the following conditions:
\begin{enumerate}
    \item $R_1=(n_1)(n_2-\ell_2+1\quad n_2-\ell_2+2\quad\cdots\quad n_2)\cdots(n_c-\ell_c+1\quad n_c-\ell_c+2\quad\cdots\quad n_c)$;
    \item $R_{n_{i-1}+k}=R_1^{k}R_{n_i}R_1^{-k}$, for each $k\in\{1,\dots,\ell_i\}$ and $i\in\{2,\dots,c\}$.
\end{enumerate}
\end{proposition}

\begin{proof}
We prove assertions \textit{1.} and \textit{2.}:
\begin{enumerate}[\itshape\,\quad1.]
    \item We  assume $R_1=(n_1)(n_2-\ell_2+1\quad n_2-\ell_2+2\quad\cdots\quad n_2)\cdots(n_c-\ell_c+1\quad n_c-\ell_c+2\quad\cdots\quad n_c)$ without loss of generality. If necessary, we may relabel indices. \textbf{This expression for $R_1$ will be assumed in the sequel.} Notice that the elements of $(X,*)$ belonging to the cycle $(n_i-\ell_i+1\quad\cdots\quad n_i)$ are precisely the elements of $L_i$ and the length of the cycle $(n_i-\ell_i+1\quad\cdots\quad n_i)$ is $\ell_i$, for each $i\in\{1,\dots,c\}$.
    \item First, by 1. and Theorem \ref{thm:equivdef} $R_{n_{i-1}+1}=R_{R_1(n_i)}=R_1 R_{n_i}R_1^{-1}$, for each $i\in\{2,\dots,c\}$. Now, if $R_{n_{i-1}+k}=R_1^k R_{n_i}R_1^{-k}$ for each $i\in\{2,\dots,c\}$ and a certain $k\in\{1,\dots,\ell_i-1\}$, then \[R_{n_{i-1}+k+1}=R_{R_1(n_{i-1}+k)}=R_1 R_{n_{i-1}+k} R_1^{-1}=R_1 R_1^k R_{n_i}R_1^{-k}R_1^{-1}=R_1^{k+1}R_{n_i}R_1^{-(k+1)}.\] Thus, we conclude, by induction, that $R_{n_{i-1}+k}=R_1^{k}R_{n_i}R_1^{-k}$, for each $k\in\{1,\dots,\ell_i\}$ and $i\in\{2,\dots,c\}$.
    \end{enumerate}
\end{proof}

\begin{definition}
Given two integers $i,j\in\mathbb{Z}^+$, we let $[i,j]$ stand for their \emph{least common multiple}.
\end{definition}

\begin{proposition}\label{prop:main}
Let $t,u,v\in\{1,\dots,c\}$ such that $x_t,x_u\in (X, *)$, with $x_t\in L_t$, $x_u\in L_u$ and $x_t*x_u\in L_v$. Then $\ell_v\mid[\ell_t,\ell_u]$.
\end{proposition}

\begin{proof}
Because $[\ell_t,\ell_u]$ is a multiple of $\ell_t$ and $x_t\in L_t$, $R_1^{[\ell_t,\ell_u]}(x_t)=x_t$. Analogously, $R_1^{[\ell_t,\ell_u]}(x_u)=x_u$. Then, \[x_t*x_u=R_1^{[\ell_t,\ell_u]}(x_t)*R_1^{[\ell_t,\ell_u]}(x_u)=R_1^{[\ell_t,\ell_u]}(x_t*x_u),\] where the last equality stems from the fact that compositions of automorphisms are automorphisms. Note that $(R_1^m(x_t*x_u)\,:\,m\in \mathbb{Z}_0^+)$ is a periodic sequence, whose period is $l_v$, since $x_t\ast x_u\in L_v$, by hypothesis. The sequence starts at $x_t*x_u$ and is formed by copies of the cycle $L_v$. So, $x_t*x_u=R_1^{[\ell_t,\ell_u]}(x_t*x_u)$ implies that $\ell_v\mid[\ell_t,\ell_u]$. This completes the proof.
\end{proof}

\begin{proposition}\label{prop:subquandle}
Let $(Y, *)$ be a subquandle of $(X, *)$. Then, $($the length of$)$ any disjoint cycle of a right translation from $(Y, *)$ is $($the length of$)$ a disjoint cycle of a right translation from $(X, *)$.
\end{proposition}
\begin{proof}
Let $y, y'\in Y\subset X$, let $m\in \mathbf{Z}^+$. Then, by induction, $R_y^m(y')\in Y$, since $(Y, *)$ is a subquandle.
Let $\ell_{y, y'}^{(Y, *)}$ be the length of the disjoint cycle of ${R_y}\big|_Y$ containing $y'$. Then, $$\ell_{y, y'}^{(Y, *)}=|\{ R_y^m(y')\, :\, m\in \mathbf{Z}^+ \}| .$$ Since $y,y'\in (X, *)$, itself a quandle, then $|\{ R_y^m(y')\, :\, m\in \mathbf{Z}^+ \}|$ is the length of the disjoint cycle of $R_y$ containing $y'$. The result follows.
\end{proof}

\begin{corollary}
Quandles of cyclic type $($Example \ref{ex:qdlesCyclic}$)$ do not have non-trivial subquandles.
\end{corollary}
\begin{proof}
Since $prof(X, *) = (1, |X|-1)$ then proper subquandles of $(X, *)$ have profiles of the sort $(1)$.
\end{proof}

\begin{definition}
Let $X$ be a finite set and $*$ a binary operation on $X$ i.e., $(X, *)$ is a finite magma. We let $\text{Aut }(X, *)$ denote the set of automorphisms on $X$. For each $\varphi\in \text{Aut }(X, *)$ we let $$\text{Fix }(\varphi) :=\{ x \in X \, :\, \varphi (x) = x \} .$$
\end{definition}

\begin{proposition}\label{prop:Fix}
Let  $(X, *)$ be a finite magma. Let $\varphi, \psi \in \text{Aut }(X, *)$. Then $\text{Fix}(\psi\varphi\psi^{-1})=\psi(\text{Fix}(\varphi))$.
\end{proposition}
\begin{proof}
Let $x\in X$. $$x\in \text{Fix}(\psi\varphi\psi^{-1})\,\,\Longleftrightarrow\,\, \psi\varphi\psi^{-1}(x)=x\,\,\Longleftrightarrow\,\, \varphi\psi^{-1}(x)=\psi^{-1}(x)\,\,\Longleftrightarrow\,\, \psi^{-1}(x) \in \text{Fix}(\varphi)\,\,\Longleftrightarrow\,\, x\in \psi(\text{Fix}(\varphi)).$$
\end{proof}

\section{Proof of Main Theorem}\label{sctn:characterization}

\textbf{In the sequel, we let $(X,*)$ stand for an SHQ of order $n$ with profile $(\ell_1,\ell_2,\dots,\ell_c)$ and $c\geq 2$, unless otherwise stated. We keep to the notation of Definition \ref{def:notation}.}\bigbreak

Proposition \ref{prop:SHQsubquandle} below can be regarded as a Corollary to Proposition \ref{prop:subquandle}. Its proof is the proof of the statement in $(3)$ in Theorem \ref{thm:theone}, but the expressions for the lengths.

\begin{proposition}\label{prop:SHQsubquandle}
$(X_i,*)$ is a subquandle of $(X,*)$, for each $i\in\{2,\dots,c\}$. Besides, $(X_i,*)$ is an SHQ of order $n_i$ with profile $(\ell_1,\ell_2,\dots,\ell_i)$, for each $i\in\{2,\dots,c\}$. Moreover, each $(X_i, *)$ is a subquandle of $(X_{i+1},*)$.
\end{proposition}

\begin{proof}
Let $i\in\{2,\dots,c\}$. For $x,y\in X_i$, we have $x\in L_t$ and $y\in L_u$, for certain $t,u\leq i$. By Proposition \ref{prop:main}, $x*y\in L_v$, for some $v$ such that $\ell_v\mid[\ell_t,\ell_u]$. Because $[\ell_t,\ell_u]\mid \ell_i$, then $v\leq i$, thus $x\ast y\in X_i$ and $(X_i,*)$ is a subquandle of $(X,*)$ $($resp., of $(X_{i+1},*))$.  Moreover, using Proposition \ref{prop:subquandle} for the cardinalities of the sets $|\{R_1^m(\cdots ) \, : \, m\in \mathbf{Z}^+    \}|$,

\begin{align*}
&1\in X_i \qquad \Longrightarrow \qquad |\{R_1^m(1) \, : \, m\in \mathbf{Z}^+    \}|=\ell_1(=1)\,\, \text{ is a length in }\,\, prof(X_i)\\
&2\in X_i\qquad \Longrightarrow \qquad |\{R_1^m(2) \, : \, m\in \mathbf{Z}^+    \}|=\ell_2\,\, \text{ (ditto)}\\
& \cdots \\
& n_j-l_j+1 \in X_i \qquad  \Longrightarrow \qquad |\{R_1^m(n_j-l_j+1) \, : \, m\in \mathbf{Z}^+    \}|=\ell_j\,\, \text{ (ditto)}\\
& \cdots  \\
& n_i-l_i+1 \in X_i\qquad \Longrightarrow  \qquad |\{R_1^m(n_i-l_i+1) \, : \, m\in \mathbf{Z}^+    \}|=\ell_i \,\, \text{ (ditto)}
\end{align*} and so, each of the $\ell_j$'s $(j=1, \dots , i)$ is the length of a disjoint cycle of a right multiplication in $(X_i,*)$. Finally, since $$|X_i|=n_i<\ell_k \quad \text{ for }\quad k>i \quad \text{$($Lemma }\ref{lem:smaller}),$$ there are no more lengths for disjoint cycles in $(X_i, *)$, according to Proposition \ref{prop:subquandle}. Then, any two disjoint cycles of a right multiplication in $(X_i, *)$ have distinct lengths, hence $(X_i, *)$ is latin (\cite{Lages_Lopes_Vojtěchovský}) hence connected. Its profile is $$prof(X_i, *)=(1, \ell_2, \dots , \ell_i)$$ and so $(X_i, *)$ is an SHQ. Since $X_i\subset X_{i+1}$, then $(X_i, *)$ is, also, a subquandle of $(X_{i+1}, *)$.
\end{proof}

\begin{corollary}\label{cor:X2}
$(X_2, *)$ is a quandle of cyclic type $($see Example \ref{ex:qdlesCyclic}$)$. In particular, there exist a prime $p$ and an $a\in\mathbf{Z}^+$ such that $\ell +1=p^a$.
\end{corollary}
\begin{proof}
This a consequence of the fact that the order of quandles of cyclic type is a prime power \cite{Vendramin_1}.
\end{proof}

In order to prove  $(2)$ in the Main Theorem, Theorem \ref{thm:theone}, we prove  Proposition \ref{prop:mainTh}.
\begin{proposition}\label{prop:mainTh}
Consider an integer $c\geq 2$ and an SHQ, $(X, *)$, with $$prof(X,*)=(1, \ell_2, \dots , \ell_c) .$$ Then, for any $2\leq i \leq c$, $(X_i, *)$ is an SHQ with $$prof(X_i, *) = (1, \ell(\ell+1)^0, \ell(\ell+1)^1, \dots \ell(\ell+1)^{i-2})$$ with $$\ell_i=\ell(\ell+1)^{i-2} .$$
\end{proposition}
\begin{proof}
By induction on $i$.
\bigbreak
(Base) For $i=2$, $(X_2,*)$ has profile $\{1, \ell\}$ for  some positive integer $\ell$. Then, $(X,*)$ is a quandle of cyclic type (\cite{Lopes_Roseman}) and so $\ell+1=p^a$ where $p$ is a prime and $a$ is a positive integer (\cite{Vendramin_1}). So $$\ell_2=\ell=\ell(\ell+1)^{2-2} .$$
\bigbreak
(Induction Step) We assume there is an integer $i\geq 2$ such that:
\begin{align*}
&\qquad(*) \qquad\text{for any }\, 2\leq j \leq i, \,\, (X_j, *) \,\, \text{ is an SHQ with} \\ \\
&\qquad \qquad prof(X_j,*)=\{1, \ell(\ell+1)^0, \ell(\ell+1)^1, \dots , \ell(\ell+1)^{j-2}\}
\end{align*}
We want to prove that
\begin{align*}
&\qquad(**) \qquad (X_{i+1}, *) \,\, \text{ is an SHQ with} \\ \\
&\qquad \qquad prof(X_{i+1},*)=\{1, \ell(\ell+1)^0, \ell(\ell+1)^1, \dots , \ell(\ell+1)^{i-2} , \ell(\ell+1)^{i-1}\}
\end{align*}
We remark that what we need to prove is that (leaning on $(*)$): $$\ell_{i+1}=\ell(\ell+1)^{i-1} .$$
So, modulo proving that $\ell_{i+1}=\ell(\ell+1)^{i-1}$, the proof of Proposition \ref{prop:mainTh} (and so, that of  $(1)$ and $(2)$ in the Main Theorem) is complete.
\bigbreak
The proof of the induction step above has a structure of its own. We, thus, deflect its proof to Subsection \ref{subsctn:proofInductionStep}, right below.
\end{proof}

\subsection{Proof of the Induction Step for $\ell_{i+1}$ in Proposition \ref{prop:mainTh}}\label{subsctn:proofInductionStep}

We remind the reader that the integer $i\geq 2$ in this Subsection has been fixed and entails $(*)$ above. We further note that this Subsection is devoted to proving that $\ell_{i+1}=\ell(\ell+1)^{i-1}$ is true, provided $(*)$ is true. This will complete the proof of Proposition \ref{prop:mainTh} and thus it will complete the proof of $(1)$ and $(2)$ in the Main Theorem (Theorem \ref{thm:theone}).

There are two parts to this proof of the induction step. They correspond to Subsubsections \ref{subsubsect:li+1dividedby} and  \ref{subsubsect:kappa},  below. In Subsubsection \ref{subsubsect:li+1dividedby} we prove
there is $\kappa\in \mathbf{Z}^+$ such that $\ell_{i+1}=\kappa\ell (\ell +1)^{i-1}$ (Lemma \ref{lem:kappa}). In Subsubsection \ref{subsubsect:kappa} we prove $\kappa = 1$ (Lemma \ref{cor:final}).
\bigbreak
We set up the following notation, valid  throughout this Subsection \ref{subsctn:proofInductionStep}.
\bigbreak

\begin{notation}
For each $x\in X_{i+1}$, we let $R_x:=R_x\vert_{X_{i+1}}$ and let $F_x:=\{y\in X_{i+1}:R_x^{\ell_i}(y)=y\}=\text{Fix }(R_x^{\ell_i})$.
\end{notation}

\subsubsection{Proof of $\ell_{i+1}=\kappa \ell (\ell+1)^{i-1}$}\label{subsubsect:li+1dividedby}

\begin{lemma}\label{lem:Fxiso}
For any $x\in X_{i+1}$,
$$(X_i, *) = (F_1, *) \cong (F_x, *)\qquad \text{ and } \qquad n_i=|F_x| .$$
\end{lemma}
\begin{proof}
Let $x\in X_{i+1}$, let $\psi\in \text{Aut }(X, *)$. Then, using Proposition \ref{prop:Fix}, $$F_{\psi(x)}=\text{Fix }(R_{\psi(x)}^{l_i})=\text{Fix }(\psi R_{x}^{l_i}\psi^{-1})=\psi (\text{Fix }(R_{x}^{l_i}))=\psi(F_x) .$$Since $\text{Aut }$ acts transitively on $(X, *)$ (since $(X, *)$ is connected) and $F_1=X_i$ subquandle of $X_{i+1}$, then each $F_x$ is a subquandle isomorphic to $F_1$. We remark that $|n_i|=|X_i|$, see Definition \ref{def:notation}.
\end{proof}
\begin{lemma}\label{lem:Fxpartition}
$\{ F_z \,:\, z\in X_{i+1} \}$ is a partition of $X_{i+1}$.
\end{lemma}
\begin{proof}
Given $y, z\in F_x$, set $C_y:=\{ R_z^m(y)\,:\, m\in \mathbf{Z}\}$. Then $C_y\subset F_x$ since $F_x$ is a subquandle (of $X_{i+1}$) isomorphic to $X_i$ so that $|C_y|=l_k\leq l_i$ so that $R_z^{l_i}(y)=y$ so that $y\in F_z$. Since $y\in F_x$ was arbitrary then $F_x\subset F_z$ so that $F_x=F_z$ since they have the same cardinality.

This proves that $\{ F_z\, : \, z\in X_{i+1}\}$ is a partition of $X_{i+1}$ (since $z\in F_x \cap F_y$ then $F_x=F_z=F_y$).
\end{proof}

\begin{lemma}\label{lem:kappa}
$$\ell_{i+1}=\kappa \ell(\ell+1)^{i-1} ,\qquad \text{ for some $\kappa\in\mathbf{Z}^+$}.$$
\end{lemma}
\begin{proof}
Since $$|F_z|=|F_1|=|X_i|=n_i=1+\sum_{j=0}^{i-2}\ell(\ell+1)^{j}=(\ell +1)^{i-1}$$ (where the one before the last equality stems from the induction hypothesis) and the $F_z$'s partition $X_{i+1}$, then $(\ell +1)^{i-1}$ divides $|X_{i+1}|=n_{i+1}=n_i+\ell_{i+1}$, so that $(\ell +1)^{i-1}$ divides $\ell_{i+1}$. Since $\gcd (\ell, \ell +1) = 1$ and $\ell$ divides $\ell_{i+1}$ (by definition of SHQ) then there exists $\kappa \in \mathbf{Z}^+$ such that $$\ell_{i+1}=\kappa \ell (\ell +1)^{i-1} .$$
\end{proof}
\subsubsection{Proof of $\kappa = 1$}\label{subsubsect:kappa}

It remains to prove that $\kappa=1$. To do so, we introduce new notation and we establish auxiliary results.

\begin{notation}
For each $x\in X_{i+1}$, we let $B_x:=\{y\in X_{i+1}:R_x^\ell(y)=y\}$.
\end{notation}

\begin{lemma}\label{lem:Bxiso}
For each $x\in X_{i+1}$,
$$(X_2, *) = (B_1, *) \cong (B_x, *) \qquad \text{ and } \qquad \ell+1=|B_x| .$$
\end{lemma}
\begin{proof}
Analogous to the proof of Lemma \ref{lem:Fxiso}.
\end{proof}

\begin{lemma}\label{lem:Bxpartition}
$\{ B_z \,:\, z\in X_{i+1} \}$ is a partition of $X_{i+1}$.
\end{lemma}
\begin{proof}
Analogous to the proof of Lemma \ref{lem:Fxpartition}.
\end{proof}

\begin{proposition}\label{prop:asspowers}
If $x,y\in X_{i+1}$ are such that $y\in B_x$, then $R_y^\ell=R_x^\ell$.
\end{proposition}
\begin{proof}Given $x\in X_{i+1}$, let $y\in B_x\setminus \{ x \}$. Then $x*y\in B_x$, since $(B_x, *)$ is a quandle. Moreover, via idempotency and left-invertibility ($X_{i+1}$ is a latin quandle) $$x\neq x*y \neq x*y'\quad \text{ for any other }y'\in B_x\setminus\{ x, y \}$$ so that $$\{ x*z\,|\, z\in B_x  \}=B_x .$$ Finally,
\begin{align*}
&R_{x*y}^{\ell}=\big( R_{R_y(x)} \big)^{\ell}=\big( R_yR_{x}R_y^{-1} \big)^{\ell}=R_yR_{x}^{\ell}R_y^{-1}=R_{R_x^{\ell}(y)}R_{x}^{\ell}R_y^{-1}=R_x^{\ell}R_yR_x^{-\ell}R_{x}^{\ell}R_y^{-1}=R_{x}^{\ell}
\end{align*}
The proof is complete.
\end{proof}

\begin{lemma}
$L_{i+1}= X_{i+1}\setminus X_i$ has an addition operation, denoted $\dot{+}$, induced by $R_1$ i.e., $x\dot{+}1=R_1(x)$, for each $x\in L_{i+1}$. Specifically, $(L_{i+1}, \dot{+})$ is a cyclic group of order $\ell_{i+1}$.
\end{lemma}
\begin{proof}
$$L_{i+1}= X_{i+1}\setminus X_i = \{n_i+1, n_i+2, \dots , n_i+\ell_{i+1}\}=n_i+\{1, 2, \dots , \ell_{i+1}\}  \overset{\Psi}{\quad
\longrightarrow \quad} \{ \overline{1}, \overline{2}, \dots , \overline{\ell}_{i+1}\}=\mathbf{Z}/(\ell_{i+1}\mathbf{Z})$$

For $n_i+s,n_i+t\in L_{i+1}$ we set $$(n_i+s)\dot{+}(n_i+t)=R_1^t(n_i+s)=\begin{cases}
n_i+s+t&\text{ if } s+t\leq \ell_{i+1}\\
n_i+s+t-\ell_{i+1}&\text{ if }s+t>\ell_{i+1}
\end{cases} .$$This is equivalent to $\overline{s}+\overline{t}=\overline{s+t}$ in $\mathbf{Z}/(\ell_{i+1}\mathbf{Z})$, via the bijection $\Psi$ above.
\end{proof}

\begin{lemma}\label{lem:BFL}
$$B_{n_{i+1}} \subset F_{n_{i+1}}\subset L_{i+1} .$$
\end{lemma}
\begin{proof}
For any $x\in X_{i+1}$,  $B_x\subset F_x$ because $\ell$ divides $\ell_i$.  Since $n_{i+1}$ belongs to both $B_{n_{i+1}}$ and $F_{n_{i+1}}$ and not to $F_1=\{1, 2, \dots , n_i\}=X_i$ then $B_{n_{i+1}}\subset F_{n_{i+1}} \subset L_{i+1}$ since the $B_j$'s (respect., the $F_k$'s) partition $X_{i+1}$.
\end{proof}

\begin{lemma}\label{lem:additivityB}
If $n_i+x, n_i+y\in B_{n_{i+1}}$, then $n_i+y-x\in B_{n_{i+1}}$, where we read $y-x$ mod $\ell_{i+1}$.
\end{lemma}
\begin{proof}
As $n_i+x,n_i+y\in B_{n_{i+1}}$, then, using Lemma \ref{lem:Bxpartition}, $R_{n_i+x}^\ell(n_i+y)=n_i+y$. Therefore, using assertion \textit{2.} in Proposition \ref{prop:assertions}, we have
\[n_i+y=R_{n_i+x}^\ell(n_i+y)=R_1^x R_{n_{i+1}}^\ell R_1^{-x}(n_i+y)=R_1^x R_{n_{i+1}}^\ell(n_i+y-x)\Leftrightarrow n_i+y-x=R_{n_{i+1}}^\ell(n_i+y-x),\]
i.e., $R_{n_{i+1}}^\ell$ fixes $n_i+y-x$, so $n_i+y-x\in B_{n_{i+1}}$ (read $y-x$ mod $\ell_{i+1}$).
\end{proof}

\begin{corollary}\label{cor:Bni+1}
$(B_{n_{i+1}},  \dot{+})$ is a $($cyclic$)$ subgroup of $(L_{i+1}, \dot{+})$, Then, the elements of $B_{n_{i+1}}$ are evenly spaced within $L_{i+1}$: $$B_{n_{i+1}}=\{n_i+j\kappa \ell (\ell+1)^{i-2}\, |\, j\in\{ 1, 2, \dots, \ell+1\}\} .$$
\end{corollary}
\begin{proof}
Since $B_{n_{i+1}}$ is closed under subtraction (Lemma \ref{lem:additivityB}) then $B_{n_{i+1}}$ is a (cyclic) subgroup of $(L_{i+1}, \dot{+})$ and the result follows, since $|B_{n_{i+1}}|=\ell +1$ and $n_{i+1}\in B_{n_{i+1}}$.
\end{proof}

Finally, we prove that $\kappa=1$ in Lemma \ref{lem:kappa}, that is, we prove that $\ell_{i+1}=\ell\cdot(\ell+1)^{i-1}$.

\begin{lemma}\label{cor:final}
$\ell_{i+1}=\ell\cdot(\ell+1)^{i-1}$.
\end{lemma}

\begin{proof}

Since $|X_{i+1}|=n_{i+1}=n_i+\ell_{i+1}$, $|F_{n_{i+1}}|=n_i$ (Lemma \ref{lem:Fxiso}) and $F_{n_{i+1}}\subset X_{i+1}$, then $|X_{i+1}\setminus F_{n_{i+1}}|=\ell_{i+1}$. Let $X_{i+1}\setminus F_{n_{i+1}}=\{g_1,g_2,\dots,g_{\ell_{i+1}}\}$. Without loss of generality,
\[R_{n_{i+1}}|_{X_{i+1}\setminus F_{n_{i+1}}}=(g_1\quad g_2\quad\dots\quad g_{\ell_{i+1}}) ,\]since $F_{n_{i+1}}$ is the union of the cycles of $R_{n_ {i+1}}$ whose lengths are less than or equal to $\ell_i$.
By assertion \textit{2.} in Proposition \ref{prop:assertions}, $R_{n_i+j\cdot\kappa\cdot\ell\cdot(\ell+1)^{i-2}}=R_1^{j\cdot\kappa\cdot\ell\cdot(\ell+1)^{i-2}}R_{n_{i+1}}R_1^{-j\cdot\kappa\cdot\ell\cdot(\ell+1)^{i-2}}$, $\forall j\in\{1,\dots,\ell+1\}$, so,
\[R_{n_i+j\cdot\kappa\cdot\ell\cdot(\ell+1)^{i-2}}|_{X_{i+1}\setminus F_{n_{i+1}}}=(R_1^{j\cdot\kappa\cdot\ell\cdot(\ell+1)^{i-2}}(g_1)\quad R_1^{j\cdot\kappa\cdot\ell\cdot(\ell+1)^{i-2}}(g_2)\quad\cdots\quad R_1^{j\cdot\kappa\cdot\ell\cdot(\ell+1)^{i-2}}(g_{\ell_{i+1}})),\]
and also, its $l$-th iterate,
\[R_{n_i+j\cdot\kappa\cdot\ell\cdot(\ell+1)^{i-2}}^\ell|_{X_{i+1}\setminus F_{n_{i+1}}}=(R_1^{j\cdot\kappa\cdot\ell\cdot(\ell+1)^{i-2}}(g_1)\quad R_1^{j\cdot\kappa\cdot\ell\cdot(\ell+1)^{i-2}}(g_{\ell+1})\quad\cdots\quad R_1^{j\cdot\kappa\cdot\ell\cdot(\ell+1)^{i-2}}(g_{\ell_{i+1}-\ell+1}))\]
\[(R_1^{j\cdot\kappa\cdot\ell\cdot(\ell+1)^{i-2}}(g_2)\quad R_1^{j\cdot\kappa\cdot\ell\cdot(\ell+1)^{i-2}}(g_{\ell+2})\quad\cdots\quad R_1^{j\cdot\kappa\cdot\ell\cdot(\ell+1)^{i-2}}(g_{\ell_{i+1}-\ell+2}))\cdots\]
\[\cdots(R_1^{j\cdot\kappa\cdot\ell\cdot(\ell+1)^{i-2}}(g_\ell)\quad R_1^{j\cdot\kappa\cdot\ell\cdot(\ell+1)^{i-2}}(g_{2\cdot\ell})\quad\cdots\quad R_1^{j\cdot\kappa\cdot\ell\cdot(\ell+1)^{i-2}}(g_{\ell_{i+1}})).\]
Since, for each $j\in \{ 1, 2, \dots , \ell +1 \}$, $$n_i+j\cdot\kappa\cdot\ell\cdot(\ell+1)^{i-2}\in B_{n_{i+1}} \qquad \qquad   (\text{Corollary \ref{cor:Bni+1}}) ,$$ then  $$R_{n_i+j\cdot\kappa\cdot\ell\cdot(\ell+1)^{i-2}}^\ell$$ does not depend on $j$ (Proposition \ref{prop:asspowers}). Since $1\in F_1$ and the $F_k$'s partition $X_{i+1}$, $1\in X_{i+1}\setminus F_{n_{i+1}}$ and we may assume that $g_\ell=1$. Thus,
\begin{align*}
(R_1^{j\cdot\kappa\cdot\ell\cdot(\ell+1)^{i-2}}(g_\ell)&\quad R_1^{j\cdot\kappa\cdot\ell\cdot (\ell+1)^{i-2}}(g_{2\cdot\ell})\quad\cdots\quad R_1^{j\cdot\kappa\cdot\ell\cdot(\ell+1)^{i-2}}(g_{\ell_{i+1}}))=\\
&=(1\quad R_1^{j\cdot\kappa\cdot\ell\cdot(\ell+1)^{i-2}}(g_{2\cdot\ell})\quad\cdots\quad R_1^{j\cdot\kappa\cdot\ell\cdot(\ell+1)^{i-2}}(g_{\ell_{i+1}}))=(1\quad g_{2\cdot\ell}'\quad g_{3\cdot\ell}'\quad\cdots\quad g_{\ell_{i+1}}')
\end{align*} does not depend on $j$.

\begin{lemma}\label{lem:GF1}
$$\{ g_{\ell}, g_{2\ell}, g_{3\ell}, \cdots  , g_{\ell_{i+1}}  \} = \{ 1, g_{2\ell}, g_{3\ell}, \cdots , g_{[\ell_{i+1}/\ell]\ell}\}\subset F_1 .$$
\end{lemma}
\begin{proof}
Assume to the contrary and suppose the inclusion is not true. Since $1\in F_1$, there exists $s\in \{ 2, \dots , \ell_{i+1}/\ell \}$ such that $$g_{s\ell}\notin F_1=\{ y\in X_{i+1}\,|\, R_1^{\ell_i}(y)=y \} .$$ Then, $g_{s\ell}\in X_{i+1}\setminus F_1$. Specifically, $g_{s\ell}$ belongs to a disjoint cycle of $R_1$ of length greater than $\ell_i$ (since the elements in $F_1$ belong to cycles whose lengths divide $\ell_i$). Since we are working within the restriction to $X_{i+1}$, there is only one such cycle. It is the cycle of length $l_{i+1}=\kappa \ell (\ell+1)^{i-1}$ (Lemma \ref{lem:kappa}). Since $j\kappa \ell(\ell +1)^{i-2}< \kappa\ell(\ell +1)^{i-1}$, for $1\leq j \leq \ell$, $$|\{ R_1^{j\kappa \ell(\ell +1)^{i-2}}(g_{s\ell})\, |\, 1\leq j \leq \ell \}|=\ell .$$  Therefore, the sequence of permutations $$\Bigg(\quad  \bigg(1\quad R_1^{j\cdot\kappa\cdot\ell\cdot(\ell+1)^{i-2}}(g_{2\cdot\ell})\quad\cdots\quad R_1^{j\kappa \ell(\ell +1)^{i-2}}(g_{s\ell})\quad\cdots\quad R_1^{j\cdot\kappa\cdot\ell\cdot(\ell+1)^{i-2}}(g_{\ell_{i+1}})\bigg)\quad |\quad 1\leq j \leq \ell+1 \quad\Bigg) $$ is not constant which contradicts Proposition \ref{prop:asspowers}. The proof of Lemma \ref{lem:GF1} is complete.

\end{proof}
We now resume the proof of Lemma \ref{cor:final}: \begin{align*}
&\{1,g_{2\cdot\ell},g_{3\cdot\ell},\dots,g_{(\ell_{i+1}/\ell)\ell}\}\subseteq F_1 \quad \Longrightarrow \quad |\{1,g_{2\cdot\ell},g_{3\cdot\ell},\dots,g_{(\ell_{i+1}/\ell)\ell}\}|\leq |F_1| \quad \Longrightarrow \quad \ell_{i+1}/\ell \leq n_i \quad \Longleftrightarrow \quad \\
&\kappa \ell(\ell+1)^{i-1}/\ell \leq (\ell+1)^{i-1} \quad \Longrightarrow \quad \kappa = 1 \qquad \text{ since }\kappa\in\mathbf{Z}^+
\end{align*}
where the equivalence above follows from Lemma \ref{lem:kappa} and from the induction hypothesis. The proof of Lemma \ref{cor:final} is complete.

\end{proof}

\subsection{The Subquandles of SHQ's are SHQ's.}\label{sctn:subquandles}

We have already proved that the $(X_i, \ast )$'s are SHQ's and subquandles of the SHQ $(X, \ast)$ (Proposition \ref{prop:SHQsubquandle}). We now prove that these are the only subquandles of $(X, \ast)$ and that their profiles have the same second smallest length.

\begin{proposition}\label{prop:subquandle2}
Let $(X, \ast)$ be an SHQ i.e., there exist $c, \ell\in \mathbf{Z}^+$ such that $$prof(X, \ast) = (1, \ell, \ell(\ell+1), \dots , \ell(\ell+1)^i, \dots , \ell(\ell+1)^{c-2}) .$$ If $(X', \ast)$ is a subquandle of $(X, \ast)$ then its profile is $$prof(X', \ast)= (1, \ell(\ell+1), \dots , \ell(\ell+1)^{c'-2}) \quad \text{ where } \quad 2\leq c' <c .$$ Furthermore, $(X', \ast)$ is isomorphic with $(X_{c'}, \ast)$.
\begin{proof}
We recall that, thanks to Proposition \ref{prop:subquandle}, the lengths of the profile of the subquandle are lengths of the profile of the quandle. We also know that the $(X_i, \ast)$'s are subquandles of $(X, \ast)$. We now prove that there cannot be a  subquandle whose profile is different from the profile of any of the $(X_i,*)$'s. Assume to the contrary and suppose $(X',*)$ is a subquandle of $(X, *)$ such that for any $i\in \{2, \dots , c\}$, $prof(X', *)\neq prof(X_i, *)$. By Proposition \ref{prop:subquandle}, $(X', *)$ is again an SHQ. Therefore, $$prof(X', *)=(1, \ell', \ell'(\ell'+1), \dots , \ell'(\ell'+1)^{c'-2})\qquad \text{ for some }\,\,\ell', c'\in \mathbf{Z}^+ .$$If $\ell'=\ell$ then $(X', *)=(X_i, *)$ for some $i\in \{2, \dots, c\}$, see Proposition \ref{prop:mainTh}. Otherwise, $$\ell'=\ell(\ell+1)^k \qquad \text{ for some }\,\,k\in\{2, \dots , c-2\} $$such that $\ell'+1$ is a prime power (e.g. $\ell=2, k=1$). We prove the latter cannot occur by separating two cases:
\begin{enumerate}[I.]
\item $prof(X', *)=(1,  \ell'(\ell'+1)^{k})$ with $k\in \{ 1, \dots , c-2\}$ (only one non-trivial length in the profile);
\item $prof(X', *)=(1,  \ell'(\ell'+1)^{k_1},  \ell'(\ell'+1)^{k_2}, \dots )$ with $k_1<k_2\in \{ 1, \dots , c-2\}$ (at least two non-trivial lengths in the profile);
\end{enumerate}

\begin{enumerate}[I.]
\item Suppose there is only one non-trivial length in the profile of $(X', \ast)$, $$prof(X', \ast)= (1, \ell(\ell+1)^k) \quad \text{ where } \quad 0\leq k \leq c-2 .$$ If $k=0$ then $(X', \ast)$'s second smallest length is $\ell$, so we assume $k>0$.
\begin{lemma}
Suppose $(X',*)$ is a subquandle of $(X, *)$ whose profile is $$prof(X',*)=(1, \ell(\ell +1)^k .$$ Then, it is isomorphic with a subquandle of $(X_{k+2},*)$ whose underlying set is $$\{ 1 \} \cup L_{k+2} .$$
\end{lemma}
\begin{proof}
Let $y_0\neq y \in (X',*)$. Then, $$\{ R_{y_0}^m(y)\,|\, m\in \mathbf{Z}  \}$$ is the cycle of length $\ell (\ell +1)^k$ in $(X',*)$. Since $(X',*)$ is latin, then there is $x\in X'$ such that $1=y_0*x$ and so $$R_x\big( \{ R_{y_0}^m(y)\,|\, m\in \mathbf{Z} \}\big)=\{ R_{y_0*x}^m(y*x)\,|\, m\in \mathbf{Z}  \}=\{ R_{1}^m(y*x)\,|\, m\in \mathbf{Z}  \}=L_{k+2}$$ since $R_x$ is an automprphism of $(X',*)$. Since $\{ 1 \} \cup L_{k+2} \subseteq X_{k+2}$, the proof is complete.
\end{proof}
Resuming the proof of Proposition \ref{prop:subquandle2}, we henceforth assume the underlying set of $(X',*)$ to be $\{ 1 \} \cup L_{k+2}$.
     Since the profile of $(X', \ast)$ is made up of two distinct lengths, then it is a latin quandle (cf. \cite{Lages_Lopes_Vojtěchovský}). Consider also the subquandle $(X_{k+2}, \ast)$ which is also latin and for which  $(X', \ast)$ is a subquandle. Then, for each $x'\in X'$, $R_{x'}(X')=X'$ which implies that $R_{x'}(X_{k+2}\setminus X')=X_{k+2}\setminus X'$ which further implies that for each $x\in X_{k+2}\setminus X'$,
$L_x(X')\subset X_{k+2}\setminus X'$ and so $|L_x(X')|\leq |X_{k+2}\setminus X'|$. Note that $|X_{k+2}|=(\ell+1)^{k+1}$ (Proposition \ref{prop:mainTh}) whereas $|X'|=\ell(\ell+1)^{k}+1$ so that $$|X_{k+2}\setminus X'|=(\ell+1)^{k+1}-\ell(\ell+1)^{k}-1=(\ell+1)^{k}-1<\ell(\ell+1)^{k}+1=|X'| ,$$but $$|X'|=|L_x(X')|\leq |X_{k+2}\setminus X'| .$$ Therefore, $k=0$.

\item Now for a profile with more than one non-trivial length. Let $$prof(X', \ast) = (1, \ell(\ell+1)^{k_1}, \ell(\ell+1)^{k_2}, \dots ) \qquad \text{ with } \quad 0< k_1 < k_2 .$$ Since the shorter lengths divide the longer lengths in the profile of the subquandle, then the subquandle is itself an SHQ. Therefore there exists $l'\in \mathbf{Z}^+$ such that
\begin{align*}
&\begin{cases}
\ell'&=\ell(\ell+1)^{k_1}\\
\ell'(\ell'+1)&=\ell(\ell+1)^{k_2}
\end{cases}\quad\Longleftrightarrow\quad
\begin{cases}
\ell'&=\ell(\ell+1)^{k_1}\\
\ell'+1&=(\ell+1)^{k_2-k_1}
\end{cases}\quad\Longleftrightarrow\quad
\begin{cases}
&\dots \\
\ell(\ell+1)^{k_1}+1&=(\ell+1)^{k_2-k_1}
\end{cases}\\
&\Longleftrightarrow\quad\begin{cases}
&\dots \\
1&=(\ell+1)^{k_2-k_1}-\ell(\ell+1)^{k_1}
\end{cases}
\end{align*}
We next cover the three different possibilities, $(i)\,\, k_2-k_1=k_1$; $(ii)\,\, k_2-k_1>k_1$; $(iii)\,\, k_2-k_1<k_1$:
\begin{enumerate}[(i)]
\item If $k_2-k_1=k_1$ then $$1=(\ell+1)^{k_2-k_1}-\ell(\ell+1)^{k_1}=(\ell+1)^{k_1}-\ell(\ell+1)^{k_1}=(\ell+1)^{k_1}(1-\ell) \quad \Longrightarrow \quad 1-\ell>0\quad \Longrightarrow \quad 1>\ell$$ which is impossible.
\item If $k_2-k_1>k_1$ and since $k_1>0$, then $$1=(\ell+1)^{k_1}\big[ (\ell+1)^{k_2-2k_1}-\ell \big]$$which is impossible because $(\ell+1)^{k_1}>1$ and $(\ell+1)^{k_2-2k_1}-\ell \geq 1$.
\item If $k_2-k_1<k_1$ then $$1=(\ell+1)^{k_2-k_1}\big[  1 - \ell(\ell+1)^{2k_1-k_2} \big]$$ which is impossible because $(\ell+1)^{k_2-k_1}>0$ while $1 - \ell(\ell+1)^{2k_1-k_2}<0$.
\end{enumerate}
\end{enumerate}
The proof is complete.
\end{proof}
\end{proposition}

\section{Remarks, Examples, and Questions for Future Work.}\label{sctn:final}

\subsection{Remarks}\label{subsctn:rem}

We remark that the Main Theorem (Theorem \ref{thm:theone}) asserts  obstructions to the existence of quandles with certain profiles. For instance, $(1, 5)$ cannot be the profile of any quandle because $1+5=2\cdot 3$ which is not a prime power. Thus, by the Main Theorem, $$(1, 5, 5k)  \qquad \qquad (1, 5, 5k, 5kk')  \qquad \qquad (1, 5, 5k, 5kk', 5kk'k'')\qquad \qquad \cdots \qquad  \text{ for }k\leq k'\leq k''\leq \cdots \in \mathbf{Z}^+ $$ cannot be profiles of quandles for they would have a quandle with profile $(1, 5)$ for subquandle and such a quandle cannot exist. On the other hand, a quandle with profile $(1, 6)$ exists but quandles with profiles $$(1, 6, 12k)  \qquad \quad (1, 6, 12k, 12kk')  \qquad \quad (1, 6, 12k, 12kk', 12kk'k'') \qquad \quad \cdots \qquad  \text{ for }k\leq k'\leq k''\leq \cdots \in \mathbf{Z}^+ $$ cannot exist because the $\ell_i$'s do not comply with $\ell_i=\ell(\ell+1)^{i-2}$, for $i>2$.

\bigbreak

\subsection{Examples from tables in the literature.}\label{subsctn:examples}

We now provide a few examples of SHQ's by inspection of the tables available in the literature.

\begin{example}
Quandles $Q_{3,i}$, $Q_{9,j}$, $Q_{27,k}$ have profiles $(1,2)$, $(1,2,6)$, $(1,2,6,18)$, respectively, for $i\in\{1\}$, $j\in\{4,5,6\}$ and $k\in\{37,\dots,40,47\dots,52,60,61\}$, cf. \cite{Lages_Lopes_Vojtěchovský}. Since $1\mid2\mid6\mid18$, all of them are SHQ's. Indeed, taking $\ell=2$ and $c=2,3,4$, respectively, in Theorem \ref{thm:theone}, we obtain the same profiles.
\end{example}

\begin{example}
Quandles $Q_{5,i}$, $Q_{25,j}$, $i\in\{2,3\}$, $j\in\{21,\dots,30\}$, have profiles $(1,4)$, $(1,4,20)$, respectively, cf. \cite{Lages_Lopes_Vojtěchovský}. Since $1\mid4\mid20$, all these quandles are SHQ's. Indeed, taking $\ell=4$ and $c=2,3$, respectively, in Theorem \ref{thm:theone}, we obtain the same profiles.
\end{example}

\subsection{Further Examples.}\label{subsctn:furtherexamples}

\drop{\begin{example}
Cyclic quandles constitute an infinite family of SHQ's with profile as in Theorem \ref{thm:theone}, for $c=2$ and for $\ell$ such that $\ell+1$ is the power of a prime number.
\end{example}
}

\begin{theorem}\label{thm:nestedcyclic}
Quandles of cyclic type constitute an infinite family of SHQ's.
\end{theorem}
\begin{proof}
Quandles of cyclic type (\cite{Lopes_Roseman}) i.e., connected quandles whose profile is $(1, \ell)$, exist if $\ell +1=p^a$, where $p$ is a prime and $a\in \mathbf{Z}^+$ (\cite{Vendramin_1}). The proof is complete.
\end{proof}

\drop{The example above provides us with an infinite family of SHQ's whose profiles are as in Theorem \ref{thm:theone} for $c=2$. Now, for each odd prime number, we present an infinite family of SHQ's whose profiles are as in Theorem \ref{thm:theone} for $c$ as large as we want. We construct such families of SHQ's using Alexander quandles, which lie inside each other like nesting dolls.}

\begin{theorem}\label{thm:infnestedqdles}
There are infinitely many families of SHQ's, each family with infinitely many quandles.
\end{theorem}

\begin{proof}
Our proof is constructive and is based on the notion of primitive root mod $n$ i.e., the generator(s) of the multiplicative group of integers modulo $n$ (\cite{Kumanduri_Romero}). Let $p$ be an odd prime. It is known that  there is a primitive root mod $p$, call it $g$; either $g$ or $g+p$ is a primitive root mod $p^2$ and so it is a primitive root mod $p^a$ for each $a>2$ (Theorem 7.2.10 in \cite{Kumanduri_Romero}).
\begin{definition}
Let
\[
h_p=\begin{cases}
g & \text{ if $g$ is a primitive root mod $p^a$, for any $a\in \mathbf{Z}^+$};\\
g+p & \text{ if $g+p$ is a primitive root mod $p^a$, for any $a\in \mathbf{Z}^+$}
\end{cases}
\]
\end{definition} We prove the following Lemma, regarding  $h_p$.

\begin{lemma}\label{lem:g-1}
$h_p-1$ is invertible mod $p^a$, for any $a\in \mathbf{Z}^+$.
\end{lemma}
\begin{proof}
We take $g\in \{  2, 3, \dots , p-1 \}$.

Suppose $h_p=g$ is the primitive root mod $p$ which is also a primitive root mod $p^2$. Then $$0<g-1<p-2$$ and so $h_p-1=g-1$ is invertible mod $p^a$, for each $a\in \mathbf{Z}^+$.

Now assume $h_p=g+p$. Then $p<g+p<2p$ and so $p\leq g+p-1 < 2p-1$ If $p=g+p-1$ then $g=1$  and so $g$ is not a primitive root.

So, in either case, $h_p-1$ is invertible mod $p^a$, for any $a\in \mathbf{Z}^+$. The proof of Lemma \ref{lem:g-1} is complete.
\end{proof}

Resuming the proof of Theorem \ref{thm:infnestedqdles}, we will construct a family of infinitely many SHQ's for each $h_p$. Let $c\geq2$. We let $(\mathbb{Z}_{p^{c-1}},*)$ be the affine quandle (\cite{Kajiwara_Nakayama}, \cite{Lages_Lopes}) over the integers modulo $p^{c-1}$, with $a*b\equiv h_p\cdot a-(h_p-1)\cdot b$ (mod $p^{c-1})$, for all $a,b\in\mathbb{Z}$. We prove that the profile of $(\mathbb{Z}_{p^{c-1}},*)$ is equal to $(1,(p-1)\cdot p^0,\dots,(p-1)\cdot p^{c-2})$, for every $c\geq2$.\par

Because $h_p-1$ is invertible, $(\mathbb{Z}_{p^{c-1}},*)$ is latin, hence connected. Then its profile is equal to the cycle structure of any of its right translations, say $R_0$. Note $R_0^t(1)=h_p^t$, for each $t\in\mathbb{Z}^+$. Indeed, $R_0(1)=h_p\cdot1=h_p$, and if $R_0^t(1)=h_p^t$, for a given $t\in\mathbb{Z}^+$, then $R_0^{t+1}(1)=R_0(R_0^t(1))=h_p\cdot h_p^t=h_p^{t+1}$. Now, we check the profile of $(\mathbb{Z}_p,*)$ is $(1,p-1)$. Indeed $R_0(0)=0$ and as $h_p$ is a primitive root modulo $p$, the least $t\in\mathbb{Z}^+$ such that $R_0^t(1)\equiv 1\mod p\Leftrightarrow h_p^t\equiv1\mod p$ is $t=\rvert\mathbb{Z}_p^\times\rvert=p-1$. Lastly, assuming the profile of $(\mathbb{Z}_{p^{c-1}},*)$ is $(1,(p-1)\cdot p^0,\dots,(p-1)\cdot p^{c-2})$, we prove the profile of $(\mathbb{Z}_{p^c},*)$ is $(1,(p-1)\cdot p^0,\dots,(p-1)\cdot p^{c-1})$, for each $c\geq2$. We regard $(\mathbf{Z}_{p^{c-1}},*)$ as a subquandle of $(\mathbf{Z}_{p^{c}},*)$ via the embedding and $\ast$ preserving map induced by  \begin{align*}T\,:\,\, &\mathbf{Z}\longrightarrow \mathbf{Z}\\
&\,\,z\mapsto pz \qquad .\end{align*}On one hand, given $x,y\in\mathbb{Z}^+$, \[R_0(x)\equiv y\mod p^{c-1}\Rightarrow R_0(x)\cdot p\equiv y\cdot p\mod p^c\Leftrightarrow h_p\cdot x\cdot p\equiv y\cdot p\mod p^c\Leftrightarrow R_0(x\cdot p)\equiv y\cdot p\mod p^c,\] i.e., $T(z)=pz$ maps cycles of $R_0\vert_{\mathbb{Z}_{p^{c-1}}}$ onto cycles of $R_0\vert_{\mathbb{Z}_{p^c}}$. On the other hand, as $h_p$ is a primitive root modulo $p^c$, the least $t\in\mathbb{Z}^+$ such that $R_0^t(1)\equiv 1\mod p^c\Leftrightarrow h_p^t\equiv1\mod p^c$ is $t=(p-1)\cdot p^{c-1}$, as $\rvert\mathbb{Z}_{p^c}^\times\rvert=(p-1)\cdot p^{c-1}$. Thus we have proved that the profile of $(\mathbb{Z}_{p^{c}},*)$ is $(1,(p-1)\cdot p^0,\dots,(p-1)\cdot p^{c-1})$, for each $c\geq2$.\par
\end{proof}

\bigbreak

\subsection{Questions for Future Work.}

\begin{enumerate}[1.]
\item Are there  other SHQ's besides those from the families of linear  Alexander quandles above (Theorems \ref{thm:nestedcyclic} and \ref{thm:infnestedqdles})?

\item How many (non-isomorphic) SHQ's are there per cardinality of the underlying set?

\item Generalize SHQ's to quandles with the same sort of profile (i.e., the shorter lengths divide the longer ones) but repeats are allowed. What can be said about  such  quandles? Are they connected?

\end{enumerate}

There is no associated data with this work.

\end{document}